\documentclass[12pt,reqno]{amsart}
\usepackage[a4paper]{geometry}

\usepackage{amssymb,amsmath,amsthm,enumerate,bbm}

\makeatletter
\@namedef{subjclassname@2020}{\textup{2020} Mathematics Subject Classification} %for MSC codes of 2020
\makeatother

\DeclareMathOperator{\Ran}{Ran}
\DeclareMathOperator{\Ker}{Ker}

\DeclareMathOperator{\clos}{clos}
\DeclareMathOperator{\supp}{supp}
\DeclareMathOperator{\diag}{diag}
\DeclareMathOperator{\BMO}{BMO}

\DeclareMathOperator{\Span}{span}

\renewcommand{\Re}{\operatorname{Re}}
\renewcommand{\Im}{\operatorname{Im}}

\newcommand{\abs}[1]{\lvert#1\rvert}

\newcommand{\norm}[1]{\lVert#1\rVert}

\newcommand{\jap}[1]{\langle#1\rangle}
\newcommand{\Jap}[1]{\left\langle#1\right\rangle}

\newcommand{\bbR}{{\mathbb R}}
\newcommand{\bbC}{{\mathbb C}}
\newcommand{\bbZ}{{\mathbb Z}}

\newcommand{\calA}{\mathcal{A}}
\newcommand{\calB}{\mathcal{B}}

\newcommand{\calM}{\mathcal{M}}
\newcommand{\calC}{\mathcal{C}}

\newcommand{\calH}{\mathcal{H}}
\newcommand{\calP}{\mathcal{P}}

\newcommand{\1}{\mathbbm{1}}

\newcommand{\dd}{\mathrm d}
\newcommand{\ii}{\mathrm i}
\newcommand{\ee}{\mathrm e}
\newcommand{\oo}{\mathrm o}

\numberwithin{equation}{section}

\theoremstyle{plain}
\newtheorem{theorem}{\bf Theorem}[section]
\newtheorem*{theorem*}{Theorem}
\newtheorem{lemma}[theorem]{\bf Lemma}
\newtheorem{proposition}[theorem]{\bf Proposition}
\newtheorem*{proposition*}{\bf Proposition}

\theoremstyle{definition}

\newtheorem*{definition*}{\bf Definition}

\theoremstyle{remark}
\newtheorem*{remark*}{\bf Remark}
\newtheorem{remark}[theorem]{\bf Remark}
\newtheorem{example}[theorem]{\bf Example}
\newtheorem*{example*}{\bf Example}

\begin{document}

\title[Functional model for anti-linear operators]{A functional model and tridiagonalisation for symmetric anti-linear operators}

\author{Alexander Pushnitski}
\address{Department of Mathematics, King's College London, Strand, London, WC2R~2LS, U.K.}
\email{alexander.pushnitski@kcl.ac.uk}

\author{Franti\v{s}ek \v{S}tampach}
\address{Department of Mathematics, Faculty of Nuclear Sciences and Physical Engineering, Czech Technical University in Prague, Trojanova 13, 12000 Prague~2, Czech Republic.}
\email{stampfra@cvut.cz}



\date{August 2024}

\begin{abstract}
We consider the class of bounded symmetric anti-linear operators $B$ with a cyclic vector. We associate with $B$ the spectral data consisting of a~probability measure and a~function. In terms of the spectral data of $B$, we introduce a functional model operator $\mathcal{B}$ acting on a model space. We prove an anti-linear variant of the spectral theorem demonstrating that $B$ is unitarily equivalent to $\mathcal{B}$. Next, we show that $B$ is also unitarily equivalent to an anti-linear tridiagonal operator and discuss connection with orthogonal polynomials in the anti-linear setting.
\end{abstract}

\maketitle

%\setcounter{tocdepth}{1}
%\tableofcontents

%%%%%%%%%%%%%%%%%%%%%%%%%%%%%%
%%%%%%%%%%%%%%%%%%%%%%%%%%%%%%
\section{Introduction}
\label{sec.a}
%%%%%%%%%%%%%%%%%%%%%%%%%%%%%%
%%%%%%%%%%%%%%%%%%%%%%%%%%%%%

\subsection{Motivation}\label{sec.a1}

All operators in this paper are bounded and act on infinite-dimensional separable Hilbert spaces. Our motivation comes from the well-known fact that there is a unitary equivalence between the following three classes of operators:
\begin{enumerate}[(1)]
\item
self-adjoint operators with simple spectrum;
\item
operators of multiplication by the independent variable in $L^2(\mu)$, where $\mu$ is a  probability measure on $\bbR$ of bounded 
infinite support;
\item
self-adjoint Jacobi matrices. 
\end{enumerate}
The classical reference is Stone's monograph~\cite[Chp.~7, \S~3]{sto_32}.

Our aim is to prove an analogous unitary equivalence between three classes of bounded \emph{symmetric} \emph{anti-linear} operators. In Section~\ref{sec.a2} we recall the precise statement of the equivalences of (1)--(3) in the classical (linear self-adjoint) case. In Section~\ref{sec.a3} we explain our anti-linear set-up. 
In Section~\ref{sec.a4} we indicate informally the nature of our main results and describe 
the structure of the rest of the paper.

\subsection{Linear self-adjoint case}\label{sec.a2}
Let $\calH$ be a Hilbert space, with the inner product $\jap{x,y}$, linear in $x$ and anti-linear in $y$. For a bounded self-adjoint operator $A$ in $\calH$ and a unit vector $\delta\in\calH$, the spectral measure $\mu$ of $A$ associated with $\delta$ is defined by 
\begin{equation}
\jap{f(A)\delta,\delta}=\int_{\bbR}f(\lambda)\dd\mu(\lambda), \quad \forall f\in C(\bbR);
\label{aa0}
\end{equation}
$\mu$ is a probability measure on $\bbR$ with a bounded support.

Recall that $\delta\in\calH$ is a cyclic vector of $A$ if and only if 
\[
\clos\{f(A)\delta: f\in \calP\}=\calH,
\] 
where $\calP$ denotes the set of all polynomials of one real variable with complex coefficients. If $A$ has a cyclic vector, $A$ is said to have simple spectrum. Since the dimension of $\calH$ is infinite, the support of the spectral measure associated with a cyclic vector must be an infinite set. We denote by $M_\lambda$ the operator of multiplication by the independent variable in $L^2(\mu)$ and by $1$  the function in $L^2(\mu)$ identically equal to $1$. 
The operator $M_\lambda$ is the ``functional model'' of $A$ in the following sense: 

\begin{proposition}\label{prp0}
Let $A$ be a bounded self-adjoint operator in a Hilbert space $\calH$ with a cyclic vector $\delta$, and  let $\mu$ be the corresponding spectral measure. Then there exists a unitary  operator $U:\calH\to L^2(\mu)$ such that $U\delta=1$ and $UA=M_\lambda U$. 
\end{proposition}

Proposition~\ref{prp0} is of course nothing but the particular case of the spectral theorem for self-adjoint operators with simple spectrum, see for example~\cite[Thm.~VII.3]{ree-sim1_80}.

Let $a_n>0$ and $b_n\in\bbR$ be bounded sequences, known as the \emph{Jacobi parameters} in this context. The corresponding \emph{Jacobi matrix} $J=J(a,b)$ is
\begin{equation}
J=\begin{pmatrix}
b_0 & a_0 & 0 & 0 & 0 &\cdots\\
a_0 & b_1 & a_1 & 0 & 0 & \cdots\\
0 & a_1 & b_2 & a_2 & 0 &\cdots\\
0 & 0 & a_2 & b_3 & a_3 & \cdots\\
\vdots&\vdots&\vdots&\vdots&\vdots&\ddots
\end{pmatrix},
\label{a0}
\end{equation}
considered as an operator on the Hilbert space $\ell^2(\bbZ_+)$, where $\bbZ_+=\{0,1,2,\dots\}$. 
It is elementary to see that $J$ is bounded and self-adjoint on $\ell^2(\bbZ_+)$ and that $e_0=(1,0,0,\dots)\in\ell^2(\bbZ_+)$ is a cyclic vector for $J$. 

\begin{proposition}\label{prp01}
Let $A$ be a bounded self-adjoint operator in a Hilbert space $\calH$ with a cyclic element $\delta$. Then there exist bounded sequences $a_n>0$, $b_n\in\bbR$ and a~unitary operator $V:\calH\to\ell^2(\bbZ_+)$ such $V\delta=e_0$ and for the corresponding Jacobi matrix $J$ we have $VA=JV$. The sequences $a_n>0$ and $b_n\in\bbR$ are uniquely defined by the above conditions.
\end{proposition}

One sometimes refers to this statement as a \emph{tridiagonalisation} of $A$; see for example~\cite[Thm.~4.2.3]{akh_21} for the proof for unbounded self-adjoint operators with simple spectrum.

Proposition~\ref{prp01} is closely related to the theory of orthogonal polynomials. 
Let $\mu$ be the \emph{spectral measure} of $J$, defined by \eqref{aa0} with $A=J$ and $\delta=e_0$.
The Jacobi parameters $a$, $b$ of $J$ can be reconstructed from the measure $\mu$ as follows. 
Applying the Gram--Schmidt process to the monomials $1,\lambda,\lambda^2,\dots$ in $L^2(\mu)$, one obtains the sequence $\{p_n\}_{n=0}^\infty$ of normalised polynomials with $\deg p_n=n$. The polynomials $p_n$ are uniquely defined by the measure $\mu$, up to a unimodular multiplicative factor. 
It is easy to see that $p_n$ satisfy the three-term recurrence relation 
\begin{equation}
\lambda p_n(\lambda)=a_np_{n+1}(\lambda)+b_np_n(\lambda)+a_{n-1}p_{n-1}(\lambda), \quad n\geq0,
\label{a01}
\end{equation}
with the convention $a_{-1}=p_{-1}(\lambda)=0$.
Conditions $a_n>0$ and $p_{0}(\lambda)=1$ fix the choice of the unimodular factor in the definition of $p_n$. Thus, we can read off the Jacobi parameters $a_n$ and $b_n$ from \eqref{a01}.

\subsection{The anti-linear set-up}\label{sec.a3}

In what follows, $B$ is a bounded \emph{anti-linear} operator on a separable Hilbert space $\calH$, i.e. $B$ satisfies
\[
B(\alpha x+\beta y)=\overline{\alpha}Bx+\overline{\beta}By
\]
for all $x,y\in\calH$ and $\alpha,\beta\in\bbC$.
 We will assume that $B$ is \emph{symmetric}, i.e. it satisfies
\begin{equation}
\jap{Bx,y}=\jap{By,x}, \quad \forall x,y\in\calH.
\label{a1}
\end{equation}
\begin{example*}
Let $\calH=\ell^2(\bbZ_+)$ and let $\calC_{\ell^2}$ be the operator of complex conjugation $x\mapsto\overline{x}$ on $\ell^2(\bbZ_+)$. Furthermore, let $A$ be a~bounded \emph{linear} operator on $\ell^2(\bbZ_+)$. Then the anti-linear operator $A\calC_{\ell^2}$ is symmetric in the above sense if and only if the matrix of $A$ in the standard basis is symmetric, $A=A^T$. 
\end{example*}
This example can be made a little more general. Let $\calC$ be a \emph{conjugation} on a~Hilbert space $\calH$, i.e. $\calC$ is an anti-linear isometric involution ($\calC=\calC_{\ell^2}$ is a typical example).
For a bounded \emph{linear} operator $A$ on $\calH$, one says that $A$ is \emph{$\calC$-symmetric}, if $\calC A=A^*\calC$. Then elementary algebra shows that the anti-linear operator $B=A\calC$ is symmetric in the sense \eqref{a1} if and only if $A$ is $\calC$-symmetric. 
The theory of $\calC$-symmetric (also known as \emph{complex-symmetric}) operators has received increased attention lately, see~\cite{gar-put_06,gar-put_07}; we will come back to this viewpoint in Section~\ref{sec.a10}. 

Let us discuss functional calculus for symmetric anti-linear operators. 
Observe that $B^2$ is linear; moreover, by the symmetry we have
\[
\jap{B^2x,x}=\jap{B(Bx),x}=\jap{Bx,Bx}\geq0,
\]
i.e. $B^2$ is self-adjoint and positive semi-definite. More generally, $B^{2n}$ is linear self-adjoint and positive semi-definite, while $B^{2n+1}$ is anti-linear for all $n\geq0$. 
For a~polynomial $f\in\calP$ we define, as usual, 
\[
f(B)=\sum_n a_n B^n, \quad \text{ if }\quad  f(t)=\sum_n a_n t^n. 
\]
The even terms in this sum are linear, while the odd ones are anti-linear. 

Let $\delta$ be a unit vector in $\calH$. 
As in the linear case, we will say that $\delta$ is \emph{cyclic} for $B$ if 
\begin{equation}
\clos\{f(B)\delta: f\in \calP\}=\calH.
\label{a6}
\end{equation}
In more detail, this condition can be rewritten as
\begin{equation}
\clos\{f(B^2)\delta+g(B^2)B\delta: f,g\in\calP\}=\calH.
\label{a6a}
\end{equation}
\begin{remark*}
Even though this is not important for what comes next, we remark that if some element $x\in\calH$ can be represented as $x=f(B^2)\delta+g(B^2)B\delta$ with some polynomials $f$ and $g$, then the choice of these polynomials is unique in this representation.  This fact is not entirely obvious and follows from Lemma~\ref{lma.d1} below.  
\end{remark*}

\subsection{Informal description of main results}\label{sec.a4}
Let $B$ be a bounded anti-linear symmetric operator with a cyclic vector. 
It turns out that it is natural to replace one spectral measure $\mu$ (see \eqref{aa0}) with a pair $(\nu,\psi)$, where $\nu$ is a~measure and $\psi$ is a~function. We call the pair $(\nu,\psi)$ the \emph{spectral data} for $B$. In Section~\ref{sec.a5}, we define $\nu$ and $\psi$ and in Section~\ref{sec.a5half} we put them in the context of spectral theory of $B$. Furthermore, in Section~\ref{sec.a6} we define a \emph{functional model} $\calB$ (analogue of multiplication operator) in terms of the spectral data $(\nu,\psi)$.

Our first main result is \textbf{Theorem~\ref{thm1}}; it is stated and proved in Section~\ref{sec.a6a}. 
Similarly to Proposition~\ref{prp0}, it asserts that $B$ is unitarily equivalent to its functional model $\calB$. 
We end Section~\ref{sec.2} with concluding remarks and a brief review of existing literature on anti-linear operators. 

Our second main result is \textbf{Theorem~\ref{thm2}}; it is stated in Section~\ref{sec.a7} and proved in Section~\ref{sec.c}.
Similarly to Proposition~\ref{prp01}, it asserts that $B$ is unitarily equivalent to an anti-linear version of a Jacobi matrix, i.e. $B$ admits tridiagonalisation. 

Our key tool in the proof of Theorem~\ref{thm2} is the sequence of polynomials that we call \emph{anti-orthogonal} (for the want of a better term). These are the polynomials $q_n$ satisfying the three-term recurrence relation 
\[
sq_{n}^{\ast}(s)=a_nq_{n+1}(s)+b_nq_n(s)+a_{n-1}q_{n-1}(s), \quad n\geq0,
\]
with the usual conventions $a_{-1}=q_{-1}(s)=0$ and $q_0(s)=1$, where $q_{n}^{\ast}(s):=\overline{q_{n}(\bar{s})}$, i.e. $q_{n}^{\ast}$ is the polynomial $q_{n}$ with \emph{complex conjugated coefficients}. Of course, there is no difference between $q^{\ast}(s)$ and $\overline{q(s)}$ when $s$ is confined to the reals which is the case below. Here $a_n>0$ but $b_n$ may be complex! These polynomials are discussed in Section~\ref{sec.c}. 

Finally, in Section~\ref{sec.d} we discuss an explicit example.

\subsection{Acknowledgements} 
A.P. is grateful to Sergei Treil for useful discussions. F.~{\v S}. acknowledges the support of the EXPRO grant No.~20-17749X of the Czech Science Foundation. The authors are grateful to the referee for useful remarks.

%%%%%%%%%%%%%%%%%%%%%%%%%%%%%%%%%
%%%%%%%%%%%%%%%%%%%%%%%%%%%%%%%%%
\section{The spectral data and a functional model}\label{sec.2}
%%%%%%%%%%%%%%%%%%%%%%%%%%%%%%%%%
%%%%%%%%%%%%%%%%%%%%%%%%%%%%%%%%%

\subsection{The spectral data}\label{sec.a5}
In what follows, $B$ is a bounded anti-linear symmetric operator in $\calH$ with a cyclic unit vector $\delta$. As already discussed, the operator $B^2$ is (linear) positive semi-definite and therefore $\abs{B}=\sqrt{B^2}$ is well-defined in terms of the usual functional calculus for self-adjoint operators. Notice that $|B|$ and $B$ commute (because $B^2$ and $B$ commute). 

We define the probability measure $\nu$ on $[0,\infty)$ by 
\begin{equation}
\jap{f(\abs{B})\delta,\delta}=\int_0^\infty f(s)\dd\nu(s), \quad \forall f\in C(\bbR).
\label{a2}
\end{equation}
In other words, $\nu$ is the spectral measure of $|B|$ associated with the cyclic vector~$\delta$.

%%%%%%%%%%%%%%%%%
\begin{proposition}\label{prp.a3}
%%%%%%%%%%%%%%%%%
There exists a unique complex-valued function $\psi\in L^\infty(s\dd\nu(s))$ such that
\begin{equation}
\jap{f(\abs{B})B\delta,\delta}=\int_0^\infty sf(s)\psi(s)\dd\nu(s), \quad \forall f\in C(\bbR).
\label{a3}
\end{equation}
The function $\psi$ satisfies $\abs{\psi(s)}\leq1$, $\nu$-a.e. $s>0$. 
\end{proposition}
\begin{proof}
The proof follows the argument of \cite[Theorem 2.2]{pus-sta_24}. 
For $f$ vanishing near the origin, we decompose $f(s)=h(s)g(s)$, where $g(s):=s^{1/2}\abs{f(s)}^{1/2}$ and 
\[
 h(s):= \begin{cases}
 s^{-1/2}f(s)/|f(s)|^{1/2},& \mbox{ if } f(s)\neq0,\\
 0,& \mbox{ if } f(s)=0.
 \end{cases}
\]
Then we find 
\begin{equation}
\jap{f(\abs{B})B\delta,\delta}
=
\jap{g(\abs{B})h(\abs{B})B\delta,\delta}
=
\jap{h(\abs{B})B\delta,g(\abs{B})\delta}.
\label{a5}
\end{equation}
In what follows, the bar over a function, say $f$, denotes the complex conjugated function defined as $\overline{f}:s\mapsto\overline{f(s)}$. By using \eqref{a1}, we get
\begin{align*}
\norm{h(\abs{B})B\delta}^2
&=
\jap{B\delta,\overline{h}(\abs{B})h(\abs{B})B\delta}
=
\jap{B\overline{h}(\abs{B})h(\abs{B})B\delta,\delta}
\\
&=
\jap{\overline{h}(\abs{B})h(\abs{B})B^2\delta,\delta}
=
\int_0^\infty s^2\abs{h(s)}^2\dd\nu(s)
=
\int_0^\infty s\abs{f(s)}\dd\nu(s),
\\
\norm{g(\abs{B})\delta}^2
&=
\int_0^\infty \abs{g(s)}^2\dd\nu(s)
=
\int_0^\infty s\abs{f(s)}\dd\nu(s).
\end{align*}
Putting this together and using Cauchy--Schwarz for the right hand side of \eqref{a5}, we obtain the estimate
\[
\abs{\jap{f(\abs{B})B\delta,\delta}}
\leq
\int_0^\infty s\abs{f(s)}\dd\nu(s). 
%\label{a4}
\]
From here we obtain the desired representation \eqref{a3} with a unique $\psi\in L^\infty(s\dd\nu(s))$ satisfying $\abs{\psi(s)}\leq1$ for $\nu$-a.e. $s>0$. 
The proof of Proposition~\ref{prp.a3} is complete. 
\end{proof}

\begin{definition*}
We call the pair $(\nu,\psi)$, where $\nu$ is defined by~\eqref{a2} and $\psi$ by Proposition~\ref{prp.a3}, the \emph{spectral data} of $B$.
\end{definition*}
\begin{remark*}
\begin{enumerate}[1.]
\item
By a simple approximation argument, we find that \eqref{a3} also holds for any bounded Borel function $f$; we will use this fact below. 
\item
Observe that the spectral data of $B$ is not a unitary invariant and is defined with reference to a distinguished cyclic vector $\delta$. For example, replacing $\delta$ by $\ee^{\ii\alpha}\delta$ results in replacing $\psi(s)$ by $\ee^{-2\ii\alpha}\psi(s)$. 
\item
The proof of Proposition~\ref{prp.a3} does not use the cyclicity assumption, and so the spectral data $(\nu,\psi)$ can be defined for any bounded symmetric anti-linear $B$ with reference to any vector $\delta$. However, the cyclicity assumption allows us to put the spectral data in the context of the spectral theory of $B$ (more precisely, of $\abs{B}$), as explained in the next subsection.
\end{enumerate}
\end{remark*}

\subsection{The multiplicity of spectrum of $\abs{B}$}\label{sec.a5half}
The purpose of this subsection is (i) to give some substance to the claim that our spectral data is a natural object in the spectral theory of $B$ and (ii) to provide context for the construction of the functional model in the next subsection. 

For a self-adjoint operator $T$, we denote by $E_T(\Delta)$ the spectral projection of $T$ corresponding to a Borel set $\Delta\subset\bbR$. We recall that a vector $x$ is called an element of \emph{maximal type} for $T$, if for any Borel set $\Delta$ we have the implication
\[
\jap{E_T(\Delta)x,x}=0 \quad\Rightarrow\quad E_T(\Delta)=0;
\]
see~\cite[Chap.~7, Sec.~3]{bir-sol_87} for general theory.
The function $\psi\in L^\infty(s\dd\nu(s))$ is defined up to sets of $\nu$-measure zero. In the next proposition, we fix a representative $\psi$ such that $\psi(s)$ is defined for all $s>0$ and satisfies $\abs{\psi(s)}\leq1$ for all $s>0$. 

\begin{proposition}\label{prp.a4}
The vector $\delta$ is an element of maximal type for $\abs{B}$.
The spectrum of $\abs{B}$ has 
\begin{align}
\text{multiplicity one on } S_1:=\{s>0:\abs{\psi(s)}=1\},
\label{a8}
\\
\text{multiplicity two on } S_2:=\{s>0:\abs{\psi(s)}<1\}.
\label{a9}
\end{align}
If $s=0$ is an eigenvalue of $\abs{B}$, then it has multiplicity one. 
\end{proposition}
\begin{proof}
First let us prove that $\delta$ is of maximal type for $\abs{B}$. Suppose $\nu(\Delta)=0$ for some Borel set $\Delta$. Then 
\[
\norm{E_{\abs{B}}(\Delta)\delta}^2=\nu(\Delta)=0
\quad \text{ and }\quad
\norm{E_{\abs{B}}(\Delta)B\delta}^2=\int_{\Delta}s^2\dd\nu(s)=0. 
\]
Since $\abs{B}$ commutes with $B^2$, it follows that for any polynomials $f$, $g$ we have 
\[
E_{\abs{B}}(\Delta)(f(B^2)\delta+g(B^2)B\delta)
=
f(B^2)E_{\abs{B}}(\Delta)\delta+g(B^2)E_{\abs{B}}(\Delta)B\delta=0.
\]
By the cyclicity assumption \eqref{a6a}, it follows that $E_{\abs{B}}(\Delta)=0$; thus, $\delta$ if of maximal type. 

Let us prove \eqref{a8}. 
We start with a trivial but important remark: since $B$ commutes with $\abs{B}$, any spectral subspace of $\abs{B}$ (i.e. any subspace of the form $\Ran E_{\abs{B}}(\Delta)$ with a Borel set $\Delta\subset\bbR$) is invariant for $B$. 
Let $f$ be a continuous function with $\supp f\subset S_1$; consider $x=f(\abs{B})B\delta$. Let us check that $x$ is in the space 
\[
\Span_{\abs{B}}(\delta):=\clos\{p(|B|)\delta: p\in\calP\}.
\]
Take $g(s)=sf(s)\psi(s)$ and consider $y=g(\abs{B})\delta$. Then 
\[
\norm{y}^2
=\int_0^\infty s^2\abs{f(s)}^2\abs{\psi(s)}^2\dd\nu(s)
=\int_0^\infty s^2\abs{f(s)}^2\dd\nu(s)
\]
because by assumption $\abs{\psi}=1$ on $\supp f$. Further, 
\[
\norm{x}^2=\int_0^\infty s^2\abs{f(s)}^2\dd\nu(s)
\]
and finally
\[
\jap{x,y}=\jap{\overline{g}(\abs{B})f(\abs{B})B\delta,\delta}
=
\int_0^\infty s\psi(s)\overline{g}(s)f(s)\dd\nu(s)
=
\int_0^\infty s^2\abs{f(s)}^2\dd\nu(s).
\]
We conclude that 
\[
\jap{x,y}
=
\norm{x}\norm{y};
\]
it follows that $x$ and $y$ are collinear, i.e. $x\in\Span_{\abs{B}}(\delta)$ as claimed.

Let us prove \eqref{a9}. Since $\delta$ is of maximal type, it suffices to prove that for any non-zero continuous $f$ with $\supp f\subset S_2$, the element $x=f(\abs{B})B\delta$ is not in $\Span_{\abs{B}}(\delta)$. 
For any $g$ and $y=g(\abs{B})\delta$ we have, similarly to the first part of the proof, 
\begin{align*}
\norm{x}^2&=\int_0^\infty s^2\abs{f(s)}^2\dd\nu(s),
\\
\norm{y}^2&=\int_0^\infty \abs{g(s)}^2\dd\nu(s),
\\
\jap{x,y}&=\jap{f(\abs{B})B\delta,g(\abs{B})\delta}
=
\int_0^\infty f(s)\overline{g}(s)s\psi(s)\dd\nu(s),
\end{align*}
and so
\[
\abs{\jap{x,y}}
\leq
\int_0^\infty \abs{f(s)}\abs{g(s)}\abs{\psi(s)}s\dd\nu(s)
<
\int_0^\infty \abs{f(s)}\abs{g(s)}s\dd\nu(s)
\leq
\norm{x}\norm{y},
\]
by using Cauchy-Schwarz at the last step. 
It follows that $x$ and $y$ are not collinear. Since $g$ is arbitrary, we find that $x\notin\Span_{\abs{B}}(\delta)$, as required.

Finally, suppose that $s=0$ is an eigenvalue of $\abs{B}$. Observe that $\abs{B}x=0$ if and only if $Bx=0$. Let us prove that the kernel of $B$ is one-dimensional. Suppose $Bx=0$ and $x\perp \delta$; then for any polynomials $f$ and $g$ we have
\[
\jap{f(B^2)\delta+g(B^2)B\delta,x}=f(0)\jap{\delta,x}+g(0)\jap{B\delta,x}=g(0)\jap{Bx,\delta}=0,
\]
and therefore $x=0$ by the cyclicity assumption \eqref{a6a}. 
Thus, $\Ker B\cap \{\delta\}^\perp=\{0\}$. 
It follows that $\Ker B$ is one-dimensional, because in any subspace of dimension $\geq2$ one can always find non-zero elements orthogonal to a given vector. 

The proof of Proposition~\ref{prp.a4} is complete. 
\end{proof}

\begin{remark}\label{rmk.psi}
Recall that $\psi$ is defined by \eqref{a3} as an element of $L^\infty(s\dd\nu(s))$, and in particular $\psi(0)$ is undefined. 
However, if $\nu$ has a point mass at $s=0$ (i.e., if $s=0$ is an eigenvalue of $\abs{B}$), it will be convenient to formally set $\psi(0)=1$. This convention allows to rewrite Proposition~\ref{prp.a4} in a more concise form: \emph{the spectrum of $\abs{B}$ has 
\begin{align}
\text{multiplicity one on } S_1:=\{s\geq0:\abs{\psi(s)}=1\},
\label{a8a}
\\
\text{multiplicity two on } S_2:=\{s\geq0:\abs{\psi(s)}<1\}. \nonumber
%\label{a9a}
\end{align}
}
\end{remark}

\subsection{The functional model $\calB$}\label{sec.a6}
Let $\nu$ be a probability measure on $[0,\infty)$ and let $\psi\in L^\infty(\nu)$ be a function satisfying $\abs{\psi(s)}\leq1$ for $\nu$-a.e. $s>0$ and $\psi(0)=1$ (the last condition is only needed if $\nu(\{0\})>0$).
Let $L^2(\nu;\bbC^2)$ be the Hilbert space of $\bbC^2$-valued functions $f=\begin{pmatrix}f_1\\ f_2\end{pmatrix}$ on $[0,\infty)$ with the norm
\[
\norm{f}^2=\int_0^\infty (\abs{f_1(s)}^2+\abs{f_2(s)}^2)\dd\nu(s).
\]
We consider the subspace 
\[
 \calM:=\left\{ f\in L^2(\nu;\bbC^2): f_{2}\equiv0 \mbox{ on } S_{1}\right\},
\]
where the set $S_1$ is as in \eqref{a8a}.
\begin{remark*}
It is easy to see that the space $\calM$ is finite-dimensional if and only if $\nu$ is a finite linear combination of point masses, i.e. of finite support.
\end{remark*}
Next, we consider the bounded anti-linear operator $\calB$ on $\mathcal M$, 
\begin{equation}
\boxed{
(\calB f)(s):=
\begin{pmatrix}
\psi(s)&\sqrt{1-\abs{\psi(s)}^2}
\\
\sqrt{1-\abs{\psi(s)}^2}&-\overline{\psi(s)}
\end{pmatrix}
\begin{pmatrix}
s\overline{f_1(s)}\\ s\overline{f_2(s)}
\end{pmatrix}.}
\label{b1}
\end{equation}
\begin{example}\label{exa.a1}
Suppose $\abs{\psi(s)}=1$ for $\nu$-a.e. $s>0$. Then the second component of $f$ disappears, and after a trivial change of notation our model reduces to the scalar one: 
\[
(\calB f)(s)=s\psi(s)\overline{f(s)}, \quad f\in L^2(\nu).
\]
This expression shows the analogy between $\calB$ and the multiplication operator $M_s$ (the operator of multiplication by the independent variable). 
\end{example}

We return to the general case.
Denote $\1:=\begin{pmatrix}1\\0\end{pmatrix}\in\mathcal M$. 
Let us discuss the spectral properties of $\calB$. 

\begin{proposition}\label{prp.3}
The operator $\mathcal B$ satisfies:
\begin{enumerate}[\rm (i)]
\item
$\calB$ is symmetric, i.e. 
$\jap{\calB f,g}=\jap{\calB g,f}$ for all $f,g\in\mathcal M$;
\item
$\abs{\calB}=M_s$, the operator of multiplication by the independent variable $s$ in $\mathcal M$;
\item
$\calB$ satisfies the identities
\begin{align*}
\jap{f(\abs{\calB})\1,\1}&=\int_0^\infty f(s)\dd\nu(s), \quad\quad\quad \forall f\in C(\bbR),
\\
\jap{f(\abs{\calB})\calB\1,\1}&=\int_0^\infty sf(s)\psi(s)\dd\nu(s), \quad \forall f\in C(\bbR);
\end{align*}
\item
$\1$ is a cyclic element for $\calB$. 
\end{enumerate}
\end{proposition}
\begin{proof}
Claim (i) follows from the symmetry of the matrix  in \eqref{b1}. 

Claim (ii) follows from the calculation
\[
\begin{pmatrix}
\psi(s)&\sqrt{1-\abs{\psi(s)}^2}
\\
\sqrt{1-\abs{\psi(s)}^2}&-\overline{\psi(s)}
\end{pmatrix}
\begin{pmatrix}
\overline{\psi(s)}&\sqrt{1-\abs{\psi(s)}^2}
\\
\sqrt{1-\abs{\psi(s)}^2}&-\psi(s)
\end{pmatrix}
=
\begin{pmatrix}
1&0\\ 0&1
\end{pmatrix}.
\]

Claim (iii) follows from (ii) and the explicit form of the matrix in \eqref{b1}.

Claim (iv). Let $f$, $g$ be polynomials; then 
\[
f(\calB^2)\1=f(M_s^2)\1=\begin{pmatrix}f(s^2)\\0\end{pmatrix}
\]
and similarly
\[
(g(\calB^2)\calB\1)(s)=\begin{pmatrix}sg(s^2)\psi(s)\\sg(s^2)\sqrt{1-\abs{\psi(s)}^2}\end{pmatrix}.
\]
It is clear that when $f$, $g$ vary over the set of all polynomials, the linear combinations of the above elements form a dense subset of $\calM$, i.e. we obtain \eqref{a6a}. 
\end{proof}

Thus, we see that the properties of $\calB$ mirror those of $B$, with $\delta=\1$. 
In particular, the spectral data of $\calB$ is $(\nu,\psi)$. 

\begin{remark*}
One can replace the matrix
in \eqref{b1} by
\[
\begin{pmatrix}
\psi(s)&\sqrt{1-\abs{\psi(s)}^2}\ee^{\ii\alpha(s)}
\\
\sqrt{1-\abs{\psi(s)}^2}\ee^{\ii\alpha(s)}&-\overline{\psi}(s)\ee^{2\ii\alpha(s)}
\end{pmatrix}
\]
with any real-valued function $\alpha(s)$. Then the resulting operator $\calB$ will be unitarily equivalent to $\calB_{\alpha=0}$ through multiplication by 
\[
\begin{pmatrix}
1&0\\0&\ee^{\ii\alpha(s)}
\end{pmatrix}.
\]
\end{remark*}

\subsection{$B$ is unitarily equivalent to $\calB$}\label{sec.a6a}
Our first main result is 
\begin{theorem}\label{thm1}
Let $B$ be a bounded anti-linear symmetric operator on a Hilbert space $\calH$ with a cyclic vector $\delta$ and the spectral data $(\nu,\psi)$. Further, let $\calB$ be the model operator \eqref{b1} in $\calM$, constructed from this spectral data. Then there exists a unitary map $U:\calH\to\calM$ such that $U\delta=\1$ and $UB=\calB U$. 
\end{theorem}
\begin{proof}
We set 
\begin{equation}
Ux=f(\calB)\1,\quad \text{ if }\quad x=f(B)\delta, \quad f\in\calP.
\label{c3}
\end{equation}
By the cyclicity assumption \eqref{a6}, this defines $Ux$ on a dense set in $\calH$. 
It will follow from Lemma~\ref{lma.d1} below that the choice of $f$ in the representation $x=f(B)\delta$ is unique. However, in order to avoid a circular argument, we will not rely on this fact here; below we check that the definition \eqref{c3} of $U x$ is independent of the choice of $f$. First let us compute the norm of $x$ in $\calH$ and the norm of  $U x$  in $\calM$. 
Write 
\[
f(s)=u(s^2)+sv(s^2),
\]
with some polynomials $u$ and $v$. Then 
\[
x=u(B^2)\delta+v(B^2)B\delta, \quad
Ux=u(\calB^2)\1+v(\calB^2)\calB\1.
\]
Using the symmetry of $B$ and the definition \eqref{a2}, \eqref{a3} of the spectral data, we find
\begin{align*}
\norm{x}^2&=
\norm{u(B^2)\delta}^2+\norm{v(B^2)B\delta}^2
+2\Re\jap{v(B^2)B\delta,u(B^2)\delta}
\\
&=\int_0^\infty\abs{u(s^2)}^2\dd\nu(s)+\int_0^\infty \abs{v(s^2)}^2s^2\dd\nu(s)
+2\Re\int_0^\infty \overline{u}(s^2)v(s^2)s\psi(s)\dd\nu(s).
\end{align*}
In a similar way, using Proposition~\ref{prp.3}(iii), we find
\begin{align*}
\norm{U x}^2&=
\norm{u(\calB^2)\1}^2+\norm{v(\calB^2)\calB\1}^2
+2\Re\jap{v(\calB^2)\calB\1,u(\calB^2)\1}
\\
&=\int_0^\infty\abs{u(s^2)}^2\dd\nu(s)+\int_0^\infty \abs{v(s^2)}^2s^2\dd\nu(s)
+2\Re\int_0^\infty \overline{u}(s^2)v(s^2)s\psi(s)\dd\nu(s)
\end{align*}
and so finally
\[
\norm{U x}=\norm{x}.
\]
In particular, if $x=0$, then $U x=0$. This shows that the definition \eqref{c3} is independent of the choice of $f$ in the representation $x=f(B)\delta$ and that $U$ is an isometry. 
By Proposition~\ref{prp.3}(iv), the map $U$ is onto $\calM$. Taking $f=1$, we find that $U \delta=\1$.

If $x=f(B)\delta$, we have
\[
Bx=Bf(B)\delta\quad\text{ and }\quad \calB U x=\calB f(\calB)\1;
\]
now it is clear that $U Bx=\calB U x$, as required. 
\end{proof}

\begin{remark*}
Let $B_1$, $B_2$ be two bounded anti-linear symmetric operators in $\calH_1$ and $\calH_2$ with cyclic elements $\delta_1$ and $\delta_2$ respectively. Suppose that there is a unitary map $\Phi: \calH_1\to\calH_2$ with $\Phi\delta_1=\delta_2$ and $\Phi B_1=B_2\Phi$. Then the spectral data of $B_1$ and $B_2$ coincide. From Theorem~\ref{thm1} we get the converse of this: if $B_1$ and $B_2$ have the same spectral data, then there exists a unitary map $\Phi$ as above. 
\end{remark*}

\subsection{Return to the complex-symmetric viewpoint}\label{sec.a10}
Let us come back to the complex-symmetric viewpoint discussed above in Section~\ref{sec.a3}: let $B=A\calC_{\calH}$, where $\calC_{\calH}$ is a conjugation in $\calH$, and $A$ is $\calC_{\calH}$-symmetric. Analogously to Theorem~\ref{thm1}, one would want to construct a functional model for $A$ (rather than $B$); in fact, this is more or less explicitly asked in the third question of \cite{gar_12}. Unfortunately, it is not clear (at least to the authors of this paper) how to achieve this on the basis of Theorem~\ref{thm1}. In order to explain this point, let us represent our model operator $\calB$ (see \eqref{b1}) as a product $\calB=\calA\calC_{\calM}$ of a \emph{linear} operator 
\[
(\calA f)(s)=
\begin{pmatrix}
\psi(s)&\sqrt{1-\abs{\psi(s)}^2}
\\
\sqrt{1-\abs{\psi(s)}^2}&-\overline{\psi(s)}
\end{pmatrix}
\begin{pmatrix}
s f_1(s)\\ sf_2(s)
\end{pmatrix}
\]
and the conjugation $\calC_{\calM}f=\overline{f}$ in $\calM$. 
Now from Theorem~\ref{thm1} we have
\[
B=A\calC_{\calH}=U^*\calB U
\]
and therefore
\begin{equation}
A=U^*\calA\calC_{\calM} U\calC_{\calH}=U^*\bigl(\calA \calC_{\calM} U\calC_{\calH} U^*\bigr) U,
\label{d3a}
\end{equation}
and so $A$ is unitarily equivalent to the \emph{linear} operator 
$\calA \calC_{\calM} U\calC_{\calH} U^*$ in $\calM$. However, the structure of the product of conjugations $\calC_{\calM} U\calC_{\calH} U^*$ is not explicit. 

In connection with \eqref{d3a}, we would like to recall the statement known as the Autonne--Takagi factorisation (see \cite[Section~3.0]{horn-johnson}): 
\emph{If $A$ is an $n\times n$ symmetric matrix, then $A=U\Sigma U^T$, where $U$ is unitary and $\Sigma=\diag(s_1,s_2,\dots,s_{n})$ is the diagonal matrix with the singular values of $A$ listed along the main diagonal.} 

We observe that \eqref{d3a} is somewhat reminiscent of the Autonne--Takagi factorisation. 
Indeed, suppose that $\calH=\bbC^n$ and $B=A\calC_{\bbC^n}$, where $\calC_{\bbC^n}$ is the standard complex conjugation in $\bbC^n$ and the matrix $A$ is symmetric in the usual sense: $A=A^T$. Then necessarily $\nu$ is a finite linear combination of point masses at the singular values $s_1,\dots,s_n$ of $A$. Assume further that $\abs{\psi(s_j)}=1$ for all $j$. 
Then (see Example~\ref{exa.a1}) the space $\calM$ can be naturally identified with $\bbC^n$ and the operator $\calA$ can be viewed as the diagonal operator with $s_j\psi(s_j)$ on the diagonal, where $s_j$ are the singular values of $A$. Finally, since now $\calC_{\calH}=\calC_{\calM}=\calC_{\bbC^n}$, we have $\calC_{\calM} U\calC_{\calH}=(U^*)^T$, and so \eqref{d3a} can be written as
\[
A=U^*\calA(U^*)^T,
\]
which is the Autonne--Takagi factorisation with $U^*$ in place of $U$ and with extra ``phases'' $\psi(s_j)$ on the diagonal of $\Sigma$. 

Despite this similarity, \eqref{d3a} is more specific as it takes into account information coming from the distinguished cyclic element $\delta$, which determines the phase function $\psi$.

\subsection{Further literature on anti-linear operators}
Since anti-linear operators is a rather arcane subject, it may be beneficial to the reader to include references to existing literature on this topic. Here we collect brief remarks based on a search of such literature.

\begin{itemize}
\item 
Basics of the theory of anti-linear operators are reviewed in~\cite{uhl_15}.
\item 
Anti-linear operators can be viewed as a special case of \emph{real linear} operators, see~\cite{huh-ruo_11}. From this point of view, every real linear operator can be represented as a sum of a linear and an anti-linear operator. 
\item
A very relevant reference is the paper  \cite{huh-per_14} by Huhtanen and Per\"{a}m\"{a}ki, but it will be convenient to postpone its discussion to Section~\ref{sec.hp14}.
\item Numerical ranges of anti-linear operators are studied in~\cite{kol-mul_24}.
\item
Bounded Hankel operators on the Hardy space $H^{2}(\mathbb{T})$ can be viewed as symmetric anti-linear mappings $g\mapsto PM_{a}\overline{g}$, where $P:L^{2}(\mathbb{T})\to H^{2}(\mathbb{T})$ is the Riesz projection, and $M_a$ is the operator of multiplication in $L^2(\mathbb{T})$ by a symbol $a\in\BMO(\mathbb{T})$. Our definition of the spectral data $(\nu,\psi)$ is strongly motivated by inverse spectral theory for Hankel operators, developed recently by P.~G\'erard and S.~Grellier; see \cite{gg_17} as well as the follow-up work \cite{gpt_23}. Section~1.8 of \cite{gpt_23} is particularly relevant to the viewpoint of this paper. 
\item
Another natural example of an anti-linear operator is furnished by the \emph{Friedrichs operator}, acting on the Bergman space $B^2(\Omega)$, where $\Omega\subset\bbC$ is a domain in the complex plane. The Friedrichs operator can be viewed as the map $g\mapsto P\overline{g}$, where $P:L^2(\Omega)\to B^2(\Omega)$ is the Bergman projection. 
See \cite{put-sha_00,put-sha_01} for the details. 

\item 
In~\cite{kap_82}, I.~Kaplansky proved that two symmetric bounded anti-linear operators $B_{1}$ and $B_{2}$ are unitarily equivalent if and only $B_{1}^{2}$ and $B_{2}^{2}$ are unitarily equivalent (in fact, it is proven for normal anti-linear operators). Notice that $B_{1}^{2}$ and $B_{2}^{2}$ are linear positive semi-definite operators. 

We recall that two self-adjoint operators are unitarily equivalent if and only if their spectral measures are mutually absolutely continuous and the spectral multiplicity functions coincide almost everywhere with respect to the spectral measure. We may combine Kaplansky's result with Proposition~\ref{prp.a4} to deduce a variant of this claim for symmetric anti-linear operators. As the statement may be of independent interest, we formulate it as a separate proposition.

\begin{proposition}
Let $B_1$ and $B_2$ be symmetric bounded anti-linear operators with cyclic elements and spectral data $(\nu_1,\psi_1)$ and $(\nu_2,\psi_2)$, respectively. Then $B_1$ and $B_2$ are unitarily equivalent if and only if $\nu_1$ and $\nu_2$ are mutually absolutely continuous and 
\[
\{s>0: \abs{\psi_1(s)}=1\}=\{s>0:\abs{\psi_2(s)}=1\}
\]
up to sets of $\nu_{i}$-measure zero, $i=1,2$.
\end{proposition}

\item 
The classical Weyl--von Neumann theorem (generalized by Kuroda~\cite{kur_58}) states that any linear self-adjoint operator is a sum of an operator with pure point spectrum and a compact operator of arbitrarily small $p$-th Schatten class norm with $p>1$.  An analogous statement for symmetric anti-linear operators is proved in~\cite[Thm.~4.5]{ruo_12}. 

\item 
Anti-linear operators appear in problems of mathematical physics~\cite{her-vuj_67}; for example in the Hartree–Bogolyubov theory in nuclear physics~\cite{her-vuj_68} and in the study of planar elasticity~\cite{put-sha_00,put-sha_01}. The so-called relative state operator of bipartite quantum systems is an anti-linear operator which appears when discussing quantum entanglement for example in studies of Einstein--Podolsky--Rosen states~\cite{are-var_00} and of quantum teleportation~\cite{kur-etal_03}. The time reversal operator $\mathcal{T}$ is an anti-linear operator of an essential importance in the so-called $\mathcal{PT}$-symmetric quantum mechanics, see the review~\cite{ben_07} with many references or the book~\cite{bag-etal_15}.
\end{itemize}

%%%%%%%%%%%%%%%%%%%%%%%%%%%%%%%%%%%%%%
%%%%%%%%%%%%%%%%%%%%%%%%%%%%%%%%%%%%%%
\section{Tridiagonalisation}
%%%%%%%%%%%%%%%%%%%%%%%%%%%%%%%%%%%%%%
%%%%%%%%%%%%%%%%%%%%%%%%%%%%%%%%%%%%%%

\subsection{Anti-linear Jacobi operator}\label{sec.a7}
Let $J$ be a Jacobi matrix \eqref{a0}; as before, we assume that the sequences $a_n$, $b_n$ are bounded and $a_n>0$, but we now allow $b_n$ to be complex. Then $J$ is a bounded linear operator on $\ell^2(\bbZ_+)$. We will consider the anti-linear operator $J\calC_{\ell^2}$ on $\ell^2(\bbZ_+)$, where $\calC_{\ell^2}$ is the operator of complex conjugation on $\ell^2(\bbZ_+)$.

Since the matrix of $J$ is symmetric, the operator $J\calC_{\ell^2}$ is symmetric in the sense \eqref{a1}. Furthermore, it is easy to see that $e_0=(1,0,\dots)\in \ell^2(\bbZ_+)$ is a cyclic element for $J\calC_{\ell^2}$. Thus, $J\calC_{\ell^2}$ is an operator of the class considered in Section~\ref{sec.a3}. Our second main result is 

\begin{theorem}\label{thm2}
Let $B$ be a bounded anti-linear symmetric operator on an infinite dimensional Hilbert space $\calH$ with a cyclic element $\delta$. Then there exist bounded sequences $a_n>0$ and $b_n\in\bbC$ and a unitary map $V:\calH\to\ell^2(\bbZ_+)$ such that $V\delta=e_0$ and for the corresponding Jacobi matrix $J$ we have $VB=(J\calC_{\ell^2})V$. 
The sequences $a_n$ and $b_n$ are uniquely defined by these conditions. 
\end{theorem}
Under the hypothesis of this theorem, it is clear that the spectral data of $B$ and $J\calC_{\ell^2}$ coincide.

\subsection{Anti-orthogonal polynomials}\label{sec.a8}
Similarly to the linear self-adjoint case, the determination of the Jacobi parameters $a_n>0$ and $b_n\in\bbC$ from the spectral data can be linked to orthogonal polynomials, but in a less obvious way. 
For any polynomial $p$, let us denote its even and odd parts by 
\[
p^{\ee}(s)=\frac12(p(s)+p(-s)), \quad
p^{\oo}(s)=\frac12(p(s)-p(-s)).
\]
Given the spectral data $(\nu,\psi)$, let us define the following sesquilinear form on polynomials:
\begin{equation}
[p,q]:=
\int_0^\infty 
\Jap{
\begin{pmatrix}1&\psi(s)\\ \overline{\psi(s)}&1\end{pmatrix}
\begin{pmatrix}p^{\ee}(s)\\ p^{\oo}(s)\end{pmatrix},
\begin{pmatrix}q^{\ee}(s)\\ q^{\oo}(s)\end{pmatrix}
}\dd\nu(s),
\label{d2}
\end{equation}
where $\jap{\cdot,\cdot}$ denotes the Euclidean inner product on $\bbC^{2}$. When expanded, we have
\begin{align*}
[p,q]
=&
\int_0^\infty p^{\ee}(s)\overline{q^{\ee}(s)}\dd\nu(s)
+
\int_0^\infty p^{\oo}(s)\overline{q^{\ee}(s)}\psi(s)\dd\nu(s)
\\
&+
\int_0^\infty p^{\ee}(s)\overline{q^{\oo}(s)}\overline{\psi(s)}\dd\nu(s)
+
\int_0^\infty p^{\oo}(s)\overline{q^{\oo}(s)}\dd\nu(s).
\end{align*}
If $(\nu,\psi)$ is the spectral data of a bounded anti-linear symmetric operator $B$ with a~cyclic unit vector $\delta$, then
\begin{equation}
[p,q]=\jap{p(B)\delta,q(B)\delta},
\label{eq:sesq_by_B}
\end{equation}
as it follows from formulas \eqref{a2} and \eqref{a3}.

It is a non-trivial step to check that this sesquilinear form is non-degenerate, i.e. if $[p,p]=0$, then $p=0$; we prove this statement in Lemma~\ref{lma.d1}.
Now let us apply the Gram--Schmidt process to the sequence $1,s,s^2,\dots$ with respect to this form. 
Let $\{q_n\}_{n=0}^\infty$ be the resulting sequence of normalised polynomials with $\deg q_n=n$. 
The polynomials $q_n$ are uniquely defined by the spectral data $(\nu,\psi)$, up to unimodular multiplicative factors. 
In Section~\ref{sec.c} we prove that this sequence of polynomials satisfies the three-term recurrence relation
\begin{equation}
sq_{n}^{\ast}(s)=a_nq_{n+1}(s)+b_nq_n(s)+a_{n-1}q_{n-1}(s), \quad n\geq0,
\label{d2a}
\end{equation}
with the usual convention $a_{-1}=q_{-1}(s)=0$, where $q_{n}^{\ast}$ denotes the polynomial $q_{n}$ with complex conjugated coefficients. Conditions $a_n>0$ and $q_{0}(s)=1$ fix the choice of multiplicative factors in the definition of $q_n$. Thus, similarly to the classical (linear self-adjoint) case, we can read off the Jacobi parameters from \eqref{d2a}. We will call the sequence of polynomials $q_n$, satisfying the recurrence relation \eqref{d2a}, \emph{anti-orthogonal polynomials}.

\subsection{Connection with the results of \cite{pus-sta_24}}
In our previous publication \cite{pus-sta_24}, we considered an inverse spectral problem for bounded non-self-adjoint Jacobi matrices $J$ of the same class as discussed above, i.e.  $a_n>0$ and $b_n\in\bbC$. We now explain the connection of Theorem~\ref{thm2} with the results of \cite{pus-sta_24}.

We first note that the spectral data $(\nu,\psi)$ was defined in \cite{pus-sta_24} from the \emph{linear}, rather than \emph{anti-linear} viewpoint, i.e. by the relations
\begin{align*}
\jap{f(\abs{J})e_0,e_0}&=\int_0^\infty f(s)\dd\nu(s),
\\
\jap{Jf(\abs{J})e_0,e_0}&=\int_0^\infty sf(s)\psi(s)\dd\nu(s),
\end{align*}
for all $f\in C(\bbR)$. However, this definition is equivalent to \eqref{a2}, \eqref{a3} with $B=J\calC_{\ell^2}$, because of the easily verifiable identities $\abs{B}=\abs{J^*}$ and 
\begin{align*}
\jap{f(\abs{B})e_0,e_0}&=\jap{f(\abs{J^*})e_0,e_0}=\jap{f(\abs{J})e_0,e_0},
\\
\jap{f(\abs{B})Be_0,e_0}&=\jap{f(\abs{J^*})Je_0,e_0}=\jap{Jf(\abs{J})e_0,e_0}.
\end{align*}
Next, we recall the main result of \cite{pus-sta_24}:
%%%%%%%%%%%%%%%%%%
\begin{theorem}[Pushnitski--{\v S}tampach~\cite{pus-sta_24}]\label{thm.a8}
%%%%%%%%%%%%%%%%%%
$\, $ % for an extra line
\begin{enumerate}[\rm (i)]
\item
Uniqueness: Any bounded Jacobi matrix $J$ satisfying $a_n>0$ and $b_n\in\bbC$ is uniquely determined by its spectral data $(\nu,\psi)$. 
\item
Surjectivity:
Let $\nu$ be a probability measure with a bounded infinite support in $[0,\infty)$. Let $\psi\in L^\infty(s\dd\nu(s))$ be a function such that $\abs{\psi(s)}\leq1$ for $\nu$-a.e. $s>0$. 
Then $(\nu,\psi)$ is the spectral data for a bounded Jacobi matrix $J$ satisfying  $a_n>0$ and $b_n\in\bbC$.
\end{enumerate}
\end{theorem}
The term ``surjectivity'' in (ii) refers to the surjectivity of the spectral map $J\mapsto (\nu,\psi)$. Similarly, (i) means that the spectral map is injective. In fact, Theorem~\ref{thm.a8} follows from Theorem~\ref{thm2}; let us explain this.

The last statement of Theorem~\ref{thm2} implies that the Jacobi parameters $a_n$, $b_n$ are uniquely determined by the spectral data $(\nu,\psi)$. This yields  the uniqueness statement of Theorem~\ref{thm.a8}. 

Next, let $(\nu,\psi)$ be as in the hypothesis of Theorem~\ref{thm.a8}(ii). 
Consider the corresponding ``functional model'' operator $\calB$ on $\calM$. The assumption that $\nu$ has infinite support implies that the space $\calM$ is infinite dimensional. 
Now taking $B=\calB$ in Theorem~\ref{thm2}, we see that there exists a Jacobi matrix with the spectral data $(\nu,\psi)$. This yields the statement of Theorem~\ref{thm.a8}(ii). 

To summarise: the construction of this paper gives a shorter (and perhaps more conceptual) proof of the main result of \cite{pus-sta_24}.

Furthermore, let $J$ be a Jacobi matrix of the above class, with the spectral data $(\nu,\psi)$. Then, taking $B=J\calC_{\ell^2}$ in Theorem~\ref{thm1}, we see that $J\calC_{\ell^2}$ is unitarily equivalent to the functional model $\calB$. 
Unfortunately, it is not clear from this whether it is possible to construct a functional model for the \emph{linear} non-self-adjoint Jacobi matrix $J$, see the discussion in Section~\ref{sec.a10}.

\subsection{Connection with the results of \cite{huh-per_14}}
\label{sec.hp14}
Some of our construction is very close to the one by Huhtanen and Per\"{a}m\"{a}ki \cite{huh-per_14}, which we briefly review here. In \cite{huh-per_14}, the authors use the concept of a \emph{biradial measure}; roughly speaking, a~probability measure $\rho$ in the complex plane is biradial if any circle centered at the origin intersects the support of $\rho$ at no more than two points. (For example, a measure supported on any straight line in $\bbC$ is biradial.) Given a compactly supported biradial measure $\rho$, the authors apply the  Gram--Schmidt process to the sequence of ``monomials'' $1$, $z$, $|z|^{2}$, $|z|^{2}z$, $|z|^{4}$, $|z|^{4}z$, $\dots$. The resulting sequence of ``polynomials'' satisfy the same three-term recurrence as \eqref{d2a} with some $a_n>0$ and $b_n\in\bbC$ and so they define a Jacobi matrix $J=J(a,b)$ of the same class as in our Theorem~\ref{thm2}. 
Then the authors prove an anti-linear version of Favard's theorem~\cite[Thm.~3.15]{huh-per_14}: \emph{Given bounded Jacobi parameters $a_{n}>0$ and $b_{n}\in\bbC$, there exists a biradial measure $\rho$ such that the Jacobi matrix constructed as above coincides with $J(a,b)$.} However, the measure $\rho$ is not unique, i.e. there are many biradial measures giving rise to the same Jacobi matrix.

Furthermore, the authors prove the following spectral theorem for anti-linear symmetric bounded operators $B$ with a cyclic vector~\cite[Thm.~3.16]{huh-per_14}: \emph{For any such $B$ there exists a~compactly supported symmetric biradial measure $\rho$ and a~unitary map $U:\calH\to L^{2}(\rho)$ such that} 
\[
 UBU^{-1}f(z) = z\overline{f(z)}.
\]
Again, $\rho$ is not unique. 

To summarise: we are able go further than \cite{huh-per_14} and establish uniqueness in both Theorems~\ref{thm1} and \ref{thm2} due to a judicious choice of the spectral data: the class of pairs $(\nu,\psi)$ is tighter than the class of biradial measures $\rho$. Furthermore, the connection of $(\nu,\psi)$ with the spectral theory of $B$ (see e.g. Proposition~\ref{prp.a4}) seems to be more transparent than in the case of biradial measures.

\section{Proof of Theorem~\ref{thm2}}
\label{sec.c}

\subsection{A sesquilinear form on polynomials}
In what follows, $B$ is a bounded anti-linear symmetric operator on an infinite dimensional Hilbert space $\calH$ with a cyclic element $\delta$, and $(\nu,\psi)$ is the corresponding spectral data. We define the sesquilinear form~\eqref{d2}. We would like to apply Gram--Schmidt to the sequence $1,s,s^2,\dots$ with respect to this sesquilinear form. In order to do this, we need to ensure that the form is non-degenerate. The argument below follows closely the proof of \cite[Lemma 6.1]{pus-sta_24}.

\begin{lemma}\label{lma.d1}
If $p$ is a polynomial such that $[p,p]=0$, then $p=0$. 
\end{lemma}
\begin{proof}
We start by observing that $\nu$ has infinite support. Indeed, suppose, to get a contradiction, that $\nu$ is a finite linear combination of point masses. This means that $\abs{B}$ has finitely many eigenvalues. By Proposition~\ref{prp.a4}, each eigenvalue has multiplicity $\leq2$. This means that $\calH$ is finite dimensional, contrary to our assumption. 

For $p\in\mathcal{P}$ suppose $[p,p]=0$. We denote
\begin{equation}
[p,p]=\int_0^\infty F(s)\dd\nu(s), \quad 
F(s):=
\Jap{
\begin{pmatrix}1&\psi(s)\\ \overline{\psi(s)}&1\end{pmatrix}
\begin{pmatrix}p^{\ee}(s)\\ p^{\oo}(s)\end{pmatrix},
\begin{pmatrix}p^{\ee}(s)\\ p^{\oo}(s)\end{pmatrix}
}.
\label{d12}
\end{equation}
Fix a Borel representative of $\psi$ defined for all $s>0$ such that $\abs{\psi(s)}\leq1$ for all $s>0$; if $\nu(\{0\})>0$, we also set $\psi(0)=1$ (see Remark~\ref{rmk.psi}).
With $\psi$ denoting this representative, let $F(s)$ be as in \eqref{d12} and let 
\[
S_0:=\{s\geq0: F(s)=0\}.
\]
By assumption, $S_0$ is the set of full $\nu$-measure and so is infinite.
Let 
\[
S_1:=\{s\geq0: \abs{\psi(s)}=1\},
\quad
S_2:=\{s\geq0: \abs{\psi(s)}<1\}.
\]
The Borel sets $S_1$ and $S_2$ are disjoint with $[0,\infty)=S_1\cup S_2$. 
The $2\times 2$ matrix in \eqref{d12} has rank one on $S_1$ and rank two on $S_2$. 

Let us split $S_0$ into a union of two disjoint Borel sets
\[
S_0=(S_1\cap S_0)\cup (S_2\cap S_0).
\]
At least one of these two sets is infinite. Consider two cases. 

\emph{Case 1: the set $S_2\cap S_0$ is infinite.}
Since the matrix in \eqref{d2}  positive definite on $S_2$, it follows that $p(s)=0$ for all $s\in S_2\cap S_0$. It follows that $p=0$, so in this case the proof is complete. 

\emph{Case 2: the set $S_1\cap S_0$ is infinite.}
We have 
\[
p^{\ee}(s)+\psi(s)p^{\oo}(s)=0, \quad \forall s\in S_1\cap S_0. 
\]
Consider the polynomial $w:=p^{\ee}(p^{\ee})^*-p^{\oo}(p^{\oo})^*$, 
or more explicitly
\[
w(s)=p^{\ee}(s)\overline{p^{\ee}(\bar{s})}-p^{\oo}(s)\overline{p^{\oo}(\bar{s})}.
\]
Using that $\abs{\psi(s)}=1$ on $S_1$, we find that $w(s)=0$ on $S_1\cap S_0$. 
Since $w$ is a~polynomial, it follows that $w$ identically equals to zero. 
Considering $w(\ii t)$ for $t\in\bbR$, we find
\[
0=p^{\ee}(\ii t)\overline{p^{\ee}(-\ii t)}-p^{\oo}(\ii t)\overline{p^{\oo}(-\ii t)}
=p^{\ee}(\ii t)\overline{p^{\ee}(\ii t)}+p^{\oo}(\ii t)\overline{p^{\oo}(\ii t)}
=\abs{p^{\ee}(\ii t)}^2+\abs{p^{\oo}(\ii t)}^2
\]
and therefore $p=0$. 
\end{proof}

\subsection{Anti-orthogonal polynomials}
Since $[\cdot,\cdot]$ is a non-degenerate sesquilinear form on polynomials, we can apply the Gram--Schmidt process to the sequence $1,s,s^2,\dots$ with respect to this form. 
Let $q_n$ be the sequence of normalised polynomials obtained in this way. The unimodular constant in the definition of $q_n$ is not fixed by this description; we will fix the choice of this constant later. 

Let us prove that the polynomials $q_n$ satisfy the three-term recurrence relation \eqref{d2a} with the usual conventions $q_{-1}=0$, $q_0=1$. Recall the notation $q^{\ast}$ is used for a~polynomial $q$ with complex conjugated coefficients.

\begin{lemma}\label{lem:s_sesq}
For any $n,k\in\bbZ_{+}$ we have  
\[
[sq_{n}^{\ast},q_k]=[sq_{k}^{\ast},q_n].
\]
In particular, this expression vanishes if $k<n-1$. 
\end{lemma}
\begin{proof}
Using \eqref{eq:sesq_by_B}, the symmetry \eqref{a1} of $B$ and the anti-linearity, we find 
\begin{align*}
[sq_{n}^{\ast},q_k]
&=
\jap{q_{n}^{\ast}(B)B\delta,q_k(B)\delta}
=
\jap{Bq_n(B)\delta,q_k(B)\delta} \\
&=
\jap{Bq_k(B)\delta,q_n(B)\delta}
=
\jap{q_{k}^{\ast}(B)B\delta,q_n(B)\delta}
=
[sq_{k}^{\ast},q_n].
\end{align*}
If $k<n-1$, then the right hand side vanishes because the degree of $sq_{k}^{\ast}$ is less than $n$ and $q_n$ is orthogonal to all polynomials of degree $<n$.  
\end{proof}

For a given $n$, since the degree of $sq_{n}^{\ast}$ is $n+1$, we can write 
\[
sq_{n}^{\ast}(s)=\sum_{k=0}^{n+1}\alpha_{n,k} q_k(s),
\]
where $\alpha_{n,k}=[sq_{n}^{\ast},q_{k}]$.
By Lemma~\ref{lem:s_sesq}, all coefficients $\alpha_{n,k}$ with $k<n-1$ vanish, so we have
\[
sq_{n}^{\ast}(s)=\alpha_{n,n+1}q_{n+1}(s)+\alpha_{n,n}q_n(s)+\alpha_{n,n-1}q_{n-1}(s). 
\]
Again by  Lemma~\ref{lem:s_sesq}, we have
\[
\alpha_{n,n+1}=[sq_{n}^{\ast},q_{n+1}]=[sq_{n+1}^{\ast},q_n]=\alpha_{n+1,n}.
\]
Thus, we obtain the three-term recurrence relation \eqref{d2a} with 
\begin{equation}
b_n=[sq_{n}^{\ast},q_n], \quad a_n=[sq_{n}^{\ast},q_{n+1}].
\label{d4}
\end{equation}
It is clear that $a_n\not=0$ for all $n$. 
Finally, by multiplying $q_n$ by suitable unimodular constants, we can make sure that $a_n>0$ for all $n$.

\subsection{The Jacobi matrix $J$}
First we check that the boundedness of $B$ ensures that the Jacobi parameters $a_n$, $b_n$ are bounded. 
\begin{lemma}
Let the Jacobi parameters $a_n$, $b_n$ be as in \eqref{d4}. 
Then 
\[
\abs{a_n}\leq\norm{B}, \quad \abs{b_n}\leq\norm{B}
\]
for all $n$. 
\end{lemma}
\begin{proof}
Let us prove the bound
\[
[sp^{\ast},sp^{\ast}]\leq \norm{B}^2[p,p]
\]
for any polynomial $p$. We have, using the symmetry \eqref{a1}, 
\[
[sp^{\ast},sp^{\ast}]
=
\jap{Bp(B)\delta,Bp(B)\delta}
=
\jap{B^2p(B)\delta,p(B)\delta}
\leq
\norm{B}^2\norm{p(B)\delta}^2
=\norm{B}^2[p,p],
\]
as claimed. Using this bound and the Cauchy--Schwarz inequality for $[\cdot,\cdot]$, we find
\[
\abs{b_n}^2
=\abs{[sq_{n}^{\ast},q_n]}^2
\leq
[sq_{n}^{\ast},sq_{n}^{\ast}]
[q_n,q_n]
\leq
\norm{B}^2[q_n,q_n]^2
=
\norm{B}^2
\]
because the polynomials $q_n$ are normalised. In the same way, 
\[
\abs{a_n}^2
=\abs{[sq_{n}^{\ast},q_{n+1}]}^2
\leq
[sq_{n}^{\ast},sq_{n}^{\ast}][q_{n+1},q_{n+1}]
\leq
\norm{B}^2[q_n,q_n][q_{n+1},q_{n+1}]=\norm{B}^2,
\]
as required. 
\end{proof}
Now let $J=J(a,b)$ be the Jacobi matrix with the Jacobi parameters as in \eqref{d4}. 
By the previous lemma, $J$ is bounded in $\ell^2(\bbZ_+)$. 
We denote by $e_n$, $n\in\bbZ_+$, the elements of the standard basis in $\ell^2(\bbZ_+)$. 

\begin{lemma}\label{lma.d4}
For all $m\in\bbZ_{+}$ we have 
\begin{equation}
q_m(J\calC_{\ell^2})e_0=e_m.
\label{d5}
\end{equation}
\end{lemma}
\begin{proof}
The proof proceeds by induction. The relation \eqref{d5} is evidently true for $m=0$. Suppose it is true for all $m\leq n$. Let us evaluate the polynomials on both sides of the recurrence relation \eqref{d2a} on the anti-linear operator $J\calC_{\ell^2}$ and apply to the element $e_0$: 
\[
J\calC_{\ell^2} q_n(J\calC_{\ell^2})e_0=a_nq_{n+1}(J\calC_{\ell^2})e_0+b_nq_n(J\calC_{\ell^2})e_0+a_{n-1}q_{n-1}(J\calC_{\ell^2})e_0
\]
(complex conjugation over $q_n$ has disappeared because of $\calC_{\ell^2}$ on the left).
By the induction hypothesis, this rewrites as
\[
J\calC_{\ell^2} e_n=a_nq_{n+1}(J\calC_{\ell^2})e_0+b_ne_n+a_{n-1}e_{n-1}.
\]
On the other hand, by the definition of $J$, we find
\[
J\calC_{\ell^2} e_n=Je_n=a_ne_{n+1}+b_ne_n+a_{n-1}e_{n-1}.
%\label{d7}
\]
Comparing, we find $q_{n+1}(J\calC_{\ell^2})e_0=e_{n+1}$. 
\end{proof}

\subsection{Construction of the unitary map $V$}
Let $J$ be as constructed above. 
As in the proof of Theorem~\ref{thm1}, we set
\[
Vf(B)\delta:=f(J\calC_{\ell^2})e_0
\]
for any polynomial $f$. In particular, $V\delta=e_0$. 

Let us check that $V$ is unitary. Consider the set $\{q_n(B)\delta\}_{n=0}^\infty$ in $\calH$. By the construction of the polynomials $q_n$ and \eqref{eq:sesq_by_B}, this set is orthonormal in $\calH$:
\[
\jap{q_n(B)\delta,q_m(B)\delta}
=
[q_n,q_m]=\delta_{n,m}.
\]
Furthermore, since $\deg q_{n}=n$ for all $n\in\bbZ_{+}$, we see that the linear span of the polynomials $q_n$ is the space of all polynomials. Using the cyclicity assumption \eqref{a6}, from here we find that the orthonormal set $\{q_n(B)\delta\}_{n=0}^\infty$ is complete in $\calH$. By Lemma~\ref{lma.d4}, we find 
\[
Vq_n(B)\delta=q_n(J\calC_{\ell^2})e_0=e_n,
\]
and so $V$ is an isometry onto $\ell^2(\bbZ_+)$. The intertwining relation $VB=J\calC_{\ell^2} V$ follows from the definition of $V$. 

\subsection{Uniqueness of $a_n$ and $b_n$}

Let $a_n>0$ and $b_n\in\bbC$ be a bounded sequence of Jacobi parameters, and let $J$ be the corresponding Jacobi matrix. Suppose that there exists a unitary map $V:\calH\to\ell^2(\bbZ_+)$ with $V\delta=e_0$ and $VB=(J\calC_{\ell^2})V$. We need to prove that the Jacobi parameters $a_n$ and $b_n$ are uniquely defined by these conditions. 

Let $q_n$ be the sequence of polynomials defined by the three-term recurrence relation \eqref{d2a} with the initial conditions $q_{-1}=0$, $q_0=1$. 
Clearly, $\deg q_{n}=n$.
As in Lemma~\ref{lma.d4}, one verifies by induction that
\[
q_m(J\calC_{\ell^2})e_0=e_m 
\]
for all $m\in\bbZ_{+}$.
It follows that
\begin{align}
\delta_{n,m}&=\jap{e_n,e_m}=\jap{q_n(J\calC_{\ell^2})e_0,q_m(J\calC_{\ell^2})e_0}
=\jap{q_n(J\calC_{\ell^2})V\delta,q_m(J\calC_{\ell^2})V\delta}
\notag
\\
&=\jap{Vq_n(B)\delta,Vq_m(B)\delta}
=\jap{q_n(B)\delta,q_m(B)\delta}
=[q_n,q_m],
\label{d8}
\end{align}
i.e. the sequence of polynomials $\{q_n\}_{n=0}^\infty$ is orthonormal with respect to the sesquilinear form $[\cdot,\cdot]$. This, together with the condition $\deg q_n=n$, uniquely 
determines $q_n$ up to multiplication by unimodular constants (i.e. $q_n$ can be obtained by the Gram--Schmidt process from the sequence $1,s,s^2,\dots$).
Further, we have
\[
b_n=\jap{J\calC_{\ell^2} e_n,e_n}, \quad a_n=\jap{J\calC_{\ell^2} e_n, e_{n+1}}, \quad n\geq0.
\]
In the same way as in \eqref{d8} we find
\begin{align}
a_n&=\jap{J\calC_{\ell^2} q_n(J\calC_{\ell^2})e_0, q_{n+1}(J\calC_{\ell^2})e_0}=\jap{Bq_n(B)\delta,q_{n+1}(B)\delta}
=[sq_{n}^{\ast},q_{n+1}],
\label{d10}
\\
b_n&=\jap{J\calC_{\ell^2} q_n(J\calC_{\ell^2})e_0, q_{n}(J\calC_{\ell^2})e_0}=\jap{Bq_n(B)\delta,q_{n}(B)\delta}
=[sq_{n}^{\ast},q_{n}].
\label{d11}
\end{align}
Since by assumption $a_n>0$ for all $n\in\bbZ_{+}$, we find 
\begin{equation}
[sq_{n}^{\ast},q_{n+1}]>0, \quad n\geq0.
\label{d9}
\end{equation}
The normalisation condition \eqref{d9} (together with $q_0=1$) fixes the unimodular constant in the construction of $q_n$. Thus, the sequence of polynomials $q_n$ is uniquely defined. Now \eqref{d10}, \eqref{d11} uniquely define the Jacobi parameters $a_n$ and $b_n$. 

%%%%%%%%%%%%%%%%%%%%%%%%%%%%%%%%%%%%
%%%%%%%%%%%%%%%%%%%%%%%%%%%%%%%%%%%%
\section{Example}\label{sec.d}
%%%%%%%%%%%%%%%%%%%%%%%%%%%%%%%%%%%%
%%%%%%%%%%%%%%%%%%%%%%%%%%%%%%%%%%%%

\subsection{The set-up and formulas for the spectral data}
The goal of this section is to compute explicitly the spectral data $(\nu,\psi)$ of the anti-linear tridiagonal operator $B=J\calC_{\ell^2}$, where the Jacobi matrix $J$ is given by 
\begin{equation}
 J=\begin{pmatrix}
 \omega & 1 & 0 & 0 & \dots \\
1 & 0 & 1 & 0 &  \dots \\
0 & 1 & 0 & 1 &  \dots \\
0 & 0 & 1 & 0 &  \dots \\
 \vdots & \vdots & \vdots & \vdots & \ddots \\
 \end{pmatrix}, \quad \omega\in\bbC, 
 \label{e1}
\end{equation}
i.e. our Jacobi parameters are 
\[
 b_{n}=\omega\delta_{n,0} \quad\mbox{ and }\quad a_{n}=1.
\]
In other words, $J$ is the discrete Schr{\" o}dinger operator with the Dirac delta potential supported at $0$ and a complex coupling constant; this operator has been used to prove optimality of spectral enclosures in~\cite{kre-lap-sta_22}. The cyclic vector $\delta$ coincides with~$e_{0}$.

In order to compare the formulas of this section with the rest of the paper, the reader may want to keep in mind that
\[
B^2=J\calC_{\ell^2}J\calC_{\ell^2}=JJ^*
\]
and therefore 
\[
\abs{B}=\sqrt{JJ^*}=\abs{J^*}.
\] 

Our aim is to prove the following proposition. 

\begin{proposition}\label{prp.d1}
The spectral data $(\nu,\psi)$ of the anti-linear tridiagonal operator $J\calC_{\ell^2}$ defined by \eqref{e1} with $\omega\in\bbC$, are as follows. The measure $\nu=\nu_{\mathrm{ac}}+\nu_{\mathrm{s}}$, where $\nu_{\mathrm{ac}}$ is absolutely continuous and supported on $[0,2]$ with the density
\[
\frac{\dd\nu_{\mathrm{ac}}}{\dd s}(s)=\frac{|\omega|^{2}+1}{\pi}\frac{\sqrt{4-s^{2}}}{(1+|\omega|^{2})^{2}-|\omega|^{2}s^{2}}.
\]
The singular part $\nu_{\mathrm{s}}$ is non-trivial if and only if $|\omega|>1$, in which case it reads
\[
\nu_{\mathrm{s}}=
 \frac{|\omega|^{2}-1}{|\omega|^{2}}\,\delta_{|\omega|+|\omega|^{-1}},
\]
i.e. $\nu_{\mathrm{s}}$ is a point mass at $|\omega|+|\omega|^{-1}$ with the weight $\frac{|\omega|^{2}-1}{|\omega|^{2}}$. 
The phase function $\psi$ satisfies 
\[
\psi(s)=\frac{\omega s}{1+|\omega|^{2}}
\]
for $s\in[0,2]$, and moreover, if $|\omega|>1$,
\[
\psi(|\omega|+|\omega|^{-1})=\frac{\omega}{|\omega|}.
\]
\end{proposition}

Recalling Proposition~\ref{prp.a4}, we see that, in particular, the multiplicity of the spectrum of $\abs{J^*}$ is two on the interval $[0,2]$. This should not be surprising; indeed, as is well known, for $\omega=0$, the matrix $J_{\omega=0}$ is explicitly diagonalisable and has the purely absolutely continous spectrum $[-2,2]$ of multiplicity one. Thus, $\abs{J_{\omega=0}}$ has the spectrum $[0,2]$ of multiplicity two. For a general $\omega$, the operator $JJ^*$ is a finite rank perturbation of $(J_{\omega=0})^2$, hence (by the Kato--Rosenblum theorem) the absolutely continuous spectra of these operators coincide and have the same multiplicity.

In the rest of this section, we prove Proposition~\ref{prp.d1}. In Section~\ref{sec.d6} at the end, we also relate the anti-orthogonal polynomials $q_n$ corresponding to $J$ to Chebyshev polynomials of the first and second kind.

\subsection{The resolvent of $JJ^{*}$}
Our first task is to compute the matrix elements of the resolvent $(z^2-JJ^*)^{-1}$. 
The matrix $JJ^{*}$ has the penta-diagonal structure
\[
 JJ^{*}=\begin{pmatrix}
 1+|\omega|^{2} & \omega & 1 & 0 & 0 & \dots \\
 \overline{\omega} & 2 & 0 & 1 & 0 & \dots \\
 1 & 0 & 2 & 0 & 1 & \dots \\
 0 & 1 & 0 & 2 & 0 & \dots \\
 0 & 0 & 1 & 0 & 2 & \dots \\
 \vdots & \vdots & \vdots & \vdots & \vdots & \ddots \\
 \end{pmatrix}.
\]
It is advantageous to view $JJ^{*}$ as a tridiagonal block matrix with $2\times 2$ blocks. Moreover, the top-left diagonal $2\times 2$ block can be considered as a perturbation. We write
\[
 JJ^{*}=T+2+X\oplus 0,
\]
where
\[
 T:=\begin{pmatrix}
 0 & 0 & 1 & 0 & 0 & \dots \\
 0 & 0 & 0 & 1 & 0 & \dots \\
 1 & 0 & 0 & 0 & 1 & \dots \\
 0 & 1 & 0 & 0 & 0 & \dots \\
 0 & 0 & 1 & 0 & 0 & \dots \\
 \vdots & \vdots & \vdots & \vdots & \vdots & \ddots \\
 \end{pmatrix}
 \quad\mbox{ and }\quad
 X:=\begin{pmatrix}
 |\omega|^{2}-1 & \omega \\
 \overline{\omega} & 0 
 \end{pmatrix}.
\]

First, we determine matrix entries of the resolvent of the Toeplitz matrix $T$. It is natural to transform the spectral parameter by the Joukowsky map $\xi\mapsto\xi+\xi^{-1}$ which is a bijective mapping from the punctured unit disk $0<|\xi|<1$ onto $\bbC\setminus[-2,2]$. Then it is straightforward to check that $\sigma(T)=[-2,2]$ and 
\[
R_{T}(\xi):=(\xi+\xi^{-1}-T)^{-1} 
\]
is a symmetric $2\times 2$ block matrix
\[
 R_{T}(\xi)=\begin{pmatrix}
 R_{0,0}(\xi) & R_{0,1}(\xi) & R_{0,2}(\xi) &\dots \\
 R_{1,0}(\xi) & R_{1,1}(\xi) & R_{1,2}(\xi) &\dots \\
 R_{2,0}(\xi) & R_{2,1}(\xi) & R_{2,2}(\xi) &\dots \\
 \vdots & \vdots & \vdots & \ddots 
 \end{pmatrix}
\]
with the $2\times 2$ blocks
\[
 R_{m,n}(\xi):=\xi^{|m-n|+1}\,\frac{1-\xi^{2\min(m,n)+2}}{1-\xi^{2}}
 \begin{pmatrix}
 1 & 0 \\ 0 & 1
 \end{pmatrix}
\]
and $0<|\xi|<1$ (only $R_{0,0}(\xi)$ will be needed below).

For $\xi\notin(-1,1)$ we have 
\[
(2+\xi+\xi^{-1}-JJ^{*})^{-1}=(\xi+\xi^{-1}-T-X\oplus 0)^{-1}=R_{T}(\xi)(1-(X\oplus 0)R_{T}(\xi))^{-1}.
\]
Denote by $P_{01}$ the orthogonal projection onto the two-dimensional subspace of $\ell^{2}(\bbZ_{+})$ spanned by $e_{0}$ and $e_{1}$. Since $P_{01}(X\oplus 0)=XP_{01}$ we further deduce the formula
\[
 P_{01}(2+\xi+\xi^{-1}-JJ^{*})^{-1}P_{01}=R_{0,0}(\xi)(1-XR_{0,0}(\xi))^{-1}=\left(R_{0,0}^{-1}(\xi)-X\right)^{-1}.
\]
From the explicit computation
\begin{align*}
\left(R_{0,0}^{-1}(\xi)-X\right)^{-1}&=
\begin{pmatrix}
1+\xi^{-1}-|\omega|^{2} & -\omega \\
 -\overline{\omega} & \xi^{-1} 
\end{pmatrix}^{-1}\\
&=\frac{\xi}{(1+\xi)(1-|\omega|^{2}\xi)}
\begin{pmatrix}
1 & \omega\xi \\
 \overline{\omega}\xi & 1+(1-|\omega|^{2})\xi 
\end{pmatrix},
\end{align*}
we find the two matrix entries of the resolvent of $JJ^{*}$,
\begin{equation}
(2+\xi+\xi^{-1}-JJ^{*})^{-1}_{0,0}=\frac{\xi}{(1+\xi)(1-|\omega|^{2}\xi)}
\label{eq:res_JJ*_00}
\end{equation}
and
\begin{equation}
(2+\xi+\xi^{-1}-JJ^{*})^{-1}_{0,1}=\frac{\omega\xi^{2}}{(1+\xi)(1-|\omega|^{2}\xi)},
\label{eq:res_JJ*_01}
\end{equation}
which will be essential for the proof.

\subsection{Even and odd extensions and Cauchy-Stieltjes transforms}
We recall that $B^2=J\calC_{\ell^2}J\calC_{\ell^2}=JJ^*$. 
Taking $f(s)=(z^2-s^2)^{-1}$ with $\Im z^2\not=0$ in \eqref{a2}, \eqref{a3}, we find
\begin{align}
F(z^2)&:=\jap{(z^2-JJ^*)^{-1}e_0,e_0}=\int_0^\infty \frac{\dd\nu(s)}{z^2-s^2},
\label{e2}
\\
G(z^2)&:=\jap{(z^2-JJ^*)^{-1}Je_0,e_0}=\int_0^\infty \frac{s\psi(s)\dd\nu(s)}{z^2-s^2}. \nonumber
%\label{e3}
\end{align}
Next, it is convenient to define the measure $\nu^{\ee}$ on the real line as a $\frac12\times$ the \emph{even} extension of $\nu$, viz.
\[
\nu^{\ee}:=\frac12(\nu+\widetilde\nu), 
\quad\text{ where }\quad
\widetilde\nu(\Delta):=\nu(-\Delta)
\]
for any Borel set $\Delta\subset\bbR$. 
Furthermore, let $\psi^{\oo}$ be the \emph{odd} extension of $\psi$ to the real line. Then we find 
\begin{equation}
zF(z^2)=\int_{-\infty}^\infty \frac{\dd\nu^{\ee}(s)}{z-s},
\quad
zG(z^2)=\int_{-\infty}^\infty \frac{s\psi^{\oo}(s)\dd\nu^{\ee}(s)}{z-s}.
\label{e4}
\end{equation}

Setting $z=\xi+\xi^{-1}$ in \eqref{e2} and using \eqref{eq:res_JJ*_00}, we find 
\[
F(2+\xi^2+\xi^{-2})=(2+\xi^2+\xi^{-2}-JJ^*)^{-1}_{0,0}
=\frac{\xi^2}{(1+\xi^2)(1-|\omega|^{2}\xi^2)}.
\]
Combining this with \eqref{e4}, we obtain
\begin{equation}
\int_{-\infty}^\infty\frac{\dd\nu^{\ee}(s)}{\xi+\xi^{-1}-s}
=(\xi+\xi^{-1})\frac{\xi^2}{(1+\xi^2)(1-\abs{\omega}^2\xi^2)}
=\frac{\xi}{1-\abs{\omega}^2\xi^2}=:Z_{\nu}(\xi+\xi^{-1})
\label{eq:cauchy_nu_even_example}
\end{equation}
for the Cauchy-Stieltjes transform of $\nu^{\ee}$. 

Next, observe that $Je_0=\omega e_0+e_1$, and therefore 
\begin{align*}
G(2+\xi^2+\xi^{-2})
&=\omega (2+\xi^2+\xi^{-2}-JJ^{*})^{-1}_{0,0}+(2+\xi^2+\xi^{-2}-JJ^{*})^{-1}_{0,1}
\\
&=\frac{\omega\xi^2}{1-\abs{\omega}^2\xi^2},
\end{align*}
by~\eqref{eq:res_JJ*_00}, \eqref{eq:res_JJ*_01}.
Again, combining with \eqref{e4}, we find
\begin{equation}
\int_{-\infty}^\infty\frac{s\psi^{\oo}(s)\dd\nu^{\ee}(s)}{\xi+\xi^{-1}-s}
=\frac{\omega\xi(1+\xi^2)}{1-\abs{\omega}^2\xi^2}=:Z_{\psi}(\xi+\xi^{-1})
\label{eq:cauchy_zeta_example}
\end{equation}
for the Cauchy-Stieltjes transform of the complex measure $s\psi^{\oo}(s)\dd\nu^{\ee}(s)$. 
Now we use the Stieltjes inversion formula to recover $\nu$ and $\psi$.

\subsection{The measure $\nu$}
We use the Stieltjes inversion formula to reconstruct $\nu^{\ee}$ from its Cauchy--Stieltjes transform~\eqref{eq:cauchy_nu_even_example}. 
If $|\omega|>1$, the measure $\nu^{\ee}$ has two atoms corresponding to $\xi=\pm|\omega|^{-1}$ with equal weights. One readily computes that
\[
 \nu^{\ee}(\{|\omega|+|\omega|^{-1}\})=\lim_{\xi\to|\omega|^{-1}}(\xi+\xi^{-1}-|\omega|-|\omega|^{-1})Z_{\nu}(\xi+\xi^{-1})=\frac{|\omega|^{2}-1}{2|\omega|^{2}}.
\]
It follows that the singular part of $\nu$ reads
\[
 \nu_{\text{s}}=\frac{|\omega|^{2}-1}{|\omega|^{2}}\delta_{|\omega|+|\omega|^{-1}},
\]
provided that $|\omega|>1$. If $|\omega|\leq1$, the singular part of $\nu^{\ee}$, and so of $\nu$, is absent.

For any $\omega\in\bbC$, the absolutely continuous part of $\nu^{\ee}$ is supported on $[-2,2]$ with the density 
\begin{align*}
\frac{\dd\nu_{\text{ac}}^{\ee}}{\dd s}(s)&=\frac{1}{2\pi\ii}\lim_{\varepsilon\to0+}\left(Z_{\nu}(s-\ii\varepsilon)-Z_{\nu}(s+\ii\varepsilon)\right)\\
&=\frac{1}{2\pi\ii}\lim_{\substack{\xi\to e^{\ii\phi} \\ |\xi|<1}}\left(Z_{\nu}(\xi+\xi^{-1})-Z_{\nu}(\overline{\xi}+\overline{\xi}^{-1})\right),
\end{align*}
where $s=2\cos\phi$ and $\phi\in(0,\pi)$. Substituting from~\eqref{eq:cauchy_nu_even_example}, we obtain explicitly
\[
\frac{\dd\nu_{\text{ac}}^{\ee}}{\dd s}(s)=\frac{|\omega|^{2}+1}{\pi}\frac{\sin\phi}{1-2|\omega|^{2}\cos2\phi+|\omega|^{4}}
=\frac{|\omega|^{2}+1}{2\pi}\frac{\sqrt{4-s^{2}}}{(1+|\omega|^{2})^{2}-|\omega|^{2}s^{2}}
\]
for $s\in[-2,2]$, which implies the formula
\[
\frac{\dd\nu_{\text{ac}}}{\dd s}(s)=\frac{|\omega|^{2}+1}{\pi}\frac{\sqrt{4-s^{2}}}{(1+|\omega|^{2})^{2}-|\omega|^{2}s^{2}}, \quad s\in[0,2].
\]

\subsection{The phase function $\psi$}
Similarly to the previous part, we employ the Stieltjes inversion formula to see that the absolutely continuous part of the complex measure $s\psi^{\oo}(s)\dd\nu^{\ee}(s)$ is supported on $[-2,2]$ with the density
\begin{align*}
s\psi^{\oo}(s)
\frac{\dd\nu^{\ee}_{\text{ac}}}{\dd s}(s)
&=\frac{1}{2\pi\ii}\lim_{\xi\to e^{\ii\phi}}\left(Z_{\psi}(\xi+\xi^{-1})-Z_{\psi}(\overline{\xi}+\overline{\xi}^{-1})\right)
\\
&=\frac{4\omega}{\pi}\frac{\sin\phi\cos^{2}\phi}{1-2|\omega|^{2}\cos2\phi+|\omega|^{4}},
\end{align*}
where $s=2\cos\phi$, $\phi\in(0,\pi)$. 
From here we obtain 
\[
 \psi(s)=\frac{\omega s}{1+|\omega|^{2}}
\]
for $s\in[0,2]$.
If $|\omega|>1$, a straightforward computation based on~\eqref{eq:cauchy_zeta_example} yields
\[
 \psi(|\omega|+|\omega|^{-1})=\frac{\omega}{|\omega|}.
\]
The proof of Proposition~\ref{prp.d1} is complete. \qed

\subsection{The anti-orthogonal polynomials}
\label{sec.d6}

Finally, we briefly examine the polynomials $q_{n}$ defined recursively by the equations
\[
 q_{n-1}(s)+\omega\delta_{n,0}q_{n}(s)+q_{n+1}(s)=sq_{n}^{\ast}(s), \quad n\geq0,
\]
with $q_{-1}(s)=0$ and $q_{0}(s)=1$. We show that they are expressible in terms of the classical Chebyshev polynomials of the first and second kind $T_{n}$ and $U_{n}$. 

Recall that both polynomial sequences $T_{n}$ and $U_{n}$ satisfy the recurrence
\begin{equation}
 u_{n-1}(x)-2xu_{n}(x)+u_{n+1}(x)=0, \quad n\geq1,
\label{eq:cheb_recur}
\end{equation}
and differ by the initial conditions: $T_{0}(x)=1$, $T_{1}(x)=x$ and $U_{0}(x)=1$, $U_{1}(x)=2x$; see for example~\cite[Sec.~10.11]{erd-etal_81}.
It follows that Chebyshev polynomials satisfy the well known identities
\[
 T_{n}(\cos\theta)=\cos n\theta
 \quad\mbox{ and }\quad
 U_{n}(\cos\theta)=\frac{\sin(n+1)\theta}{\sin\theta}.
\]
When expressed in slightly different terms, we get the formulas
\begin{equation}
T_{n}\left(\frac{\xi+\xi^{-1}}{2}\right)=\frac{\xi^{n}+\xi^{-n}}{2}
\quad\mbox{ and }\quad
U_{n}\left(\frac{\xi+\xi^{-1}}{2}\right)=\frac{\xi^{n+1}-\xi^{-n-1}}{\xi-\xi^{-1}}
\label{eq:cheb_id1}
\end{equation}
for all $n\in\bbZ_{+}$

\begin{proposition}
For all $n\in\bbZ_{+}$ and $\omega\in\bbC$ we have the identities
\begin{equation}
 q_{2n}(s)=\frac{2-2\overline{\omega}s}{2-s^{2}}\,T_{2n}\!\left(\frac{s}{2}\right)+\frac{2\overline{\omega}-s}{2-s^{2}}\,U_{2n+1}\!\left(\frac{s}{2}\right)
\label{eq:q_even_cheb}
\end{equation}
and
\begin{equation}
 q_{2n+1}(s)=-\frac{2\omega}{2-s^{2}}\,T_{2n}\!\left(\frac{s}{2}\right)+\frac{2+\omega s-s^{2}}{2-s^{2}}\,U_{2n+1}\!\left(\frac{s}{2}\right).
\label{eq:q_odd_cheb}
\end{equation}
\end{proposition}

\begin{proof}
We employ the Joukowsky transform once more and write $s=\xi+\xi^{-1}$ in expressions on right hand sides of \eqref{eq:q_even_cheb} and \eqref{eq:q_odd_cheb}. With the aid of~\eqref{eq:cheb_id1}, we find that these expressions then equal
\[
 A_{\xi}(\overline{\omega})\xi^{2n}+B_{\xi}(\overline{\omega})\xi^{-2n}=:r_{2n}(\xi)
\]
and
\[
 A_{\xi}(\omega)\xi^{2n+1}+B_{\xi}(\omega)\xi^{-2n-1}=:r_{2n+1}(\xi),
\]
with
\[
A_{\xi}(\omega):=\frac{\xi-\omega}{\xi-\xi^{-1}} 
\quad\mbox{ and }\quad
B_{\xi}(\omega):=\frac{\omega-\xi^{-1}}{\xi-\xi^{-1}}.
\]
It follows easily from the structure of $r_{n}(\xi)$ that the recurrence
\[
  r_{n-1}(\xi)+\omega\delta_{n,0}r_{n}(\xi)+r_{n+1}(\xi)=(\xi+\xi^{-1})\overline{r_{n}}(\xi),
\]
with $r_{-1}:=0$, holds for all $n\in\bbZ_{+}$. Since also $q_{0}=r_{0}=1$, we conclude that $q_{n}((\xi+\xi^{-1})/2)=r_{n}(\xi)$ for all $n\in\bbZ_{+}$, which implies \eqref{eq:q_even_cheb} and \eqref{eq:q_odd_cheb}.
\end{proof}

\begin{remark}
With the aid of identities 
\[
 T_{n}(2x^{2}-1)=T_{2n}(x) \quad\mbox{ and }\quad U_{2n+1}(x)=2xU_{n}(2x^{2}-1),
\]
see for example~\cite[p.~98]{sny_66}, formulas~\eqref{eq:q_even_cheb} and~\eqref{eq:q_odd_cheb} can be also written in the form
\[
 q_{2n}(s)=\frac{2-2\overline{\omega}s}{2-s^{2}}\,T_{n}\!\left(\frac{s^{2}}{2}-1\right)+\frac{2\overline{\omega}s-s^{2}}{2-s^{2}}\,U_{n}\!\left(\frac{s^{2}}{2}-1\right)
\]
and
\[
 q_{2n+1}(s)=-\frac{2\omega}{2-s^{2}}\,T_{n}\!\left(\frac{s^{2}}{2}-1\right)+\frac{2s+\omega s^{2}-s^{3}}{2-s^{2}}\,U_{n}\!\left(\frac{s^{2}}{2}-1\right).
\]
\end{remark}

\begin{remark}
If $\omega=0$, $q_{n}$ simplifies to the classical Chebyshev orthogonal polynomials, namely
\[
q_{n}(s)=U_{n}\left(\frac{s}{2}\right).
\]
This is obvious for the odd index from~\eqref{eq:q_odd_cheb}. To see that it is the case also for even indices, one needs to apply the identity
\[
T_{n}(x)+(2x^{2}-1)U_{n}(x)=xU_{n+1}(x)
\]
in~\eqref{eq:q_even_cheb}; see for example~\cite[Eqs.~(3),(4), Sec.~10.11]{erd-etal_81}.
Moreover, the spectral measure $\mu$ of $J_{\omega=0}$, when restricted to interval $[0,2]$, coincides with measure $\nu/2$.
\end{remark}

\bibliographystyle{acm}
%\bibliography{models}

\end{document}